\documentclass{article}
\usepackage{fullpage}
\usepackage{graphicx} % Required for inserting images
\usepackage{amsfonts}
\usepackage{amsthm}
\usepackage{amsmath}
\usepackage{amssymb}
\usepackage{mathabx}
\usepackage{mathrsfs}
\usepackage{mathptmx}
\usepackage{comment}
\usepackage{enumitem}
\usepackage{hyperref}
\usepackage[capitalise]{cleveref}
\usepackage{xcolor}
\usepackage[
backend=biber,
style=alphabetic,
sorting=nyt,
maxnames=99,
doi=false,
url=false,
isbn=false
]{biblatex}
\usepackage{titling}
\usepackage{tikz}
\usetikzlibrary{decorations.pathreplacing}
\usetikzlibrary{calc,arrows.meta}

\newcommand{\Perm}{\mathop{\mathrm{Perm}}}
\newcommand{\Sym}{\mathop{\mathrm{Sym}}}
\newcommand{\ds}{\mathop{d_\square}}

\newcommand{\supp}{\mathop{\mathrm{supp}}}

\newtheorem{theorem}{Theorem}[section]
\newtheorem{lemma}[theorem]{Lemma}
\newtheorem{prop}[theorem]{Proposition}

\theoremstyle{definition}

\newtheorem{remark}[theorem]{Remark}
\newtheorem{example}[theorem]{Example}

\title{Chebyshev centers and radius of the set of permutons}
\thanksmarkseries{alph}
\author{Bal\'azs Maga\thanks{HUN-REN Alfréd Rényi Institute of Mathematics, Budapest, Hungary. Email: \texttt{magab@renyi.hu}
\newline Supported by the NKFIH 152535 project, funded by the Ministry of Innovation and Technology of Hungary from the National Research, Development and Innovation Fund.}}
\date{}

\begin{document}

\maketitle
\begin{abstract}
We study the metric geometry of the set of permutons under the rectangular distance $\ds$.
We determine the Chebyshev radius to be $1/4$ and characterize all Chebyshev centers: a permuton is a center if and only if it is $1/2$-periodic in each coordinate.
We also describe permutons that attain the extremal distance $1/4$ from a given center.

\noindent \textbf{Keywords}: permutons, rectangular distance, Chebyshev center, Chebyshev radius
\end{abstract}

\section{Introduction}

A permuton is a probability measure on the unit square with uniform marginals. Permutons are the natural limit objects of permutations \cite{HOPPEN201393}, where convergence is defined via convergent substructure densities arising from sampling (cf. graphons and graph convergence). Permutons have received much attention in recent years, see for example \cite{ALON2022102361, BASSINO2022108513, GLEBOV2015112, Rahman2019, Jasinska2025}.

We denote the set of permutons by $\Perm$. The convergence of a sequence of permutations $\pi_n\in \Sym(n)$ to $\mu\in\Perm$ can be defined in essentially two equivalent ways. Following the usual approach to limit theories, the standard one is to define the density of subpermutations (or patterns) both in $\pi_n$ and in $\mu$, and say that $\pi_n$ converges to $\mu$ if all pattern densities converge as $n\to\infty$. On the other hand, this turns out to be equivalent to the fact that a sequence of permutons $\widehat{\pi_n}$ naturally associated with $\pi_n$ converges weakly to $\mu$. 

We can make both points of view quantitative. For the second one, we do so by metrizing the topology of weak convergence, for which a simple choice is the \emph{rectangular distance}, the natural analogue of the cut norm for permutation limits, defined by
$$\ds(\mu, \nu)=\max_{R=[a,b]\times [c, d]}|\mu(R)-\nu(R)|.$$
One should note that $\ds$ does not metrize the topology of weak convergence on the space of all probability measures of $[0 ,1]^2$. However, it does so restricted to permutons as they have continuous cumulative distribution functions.

The set $\Perm$ is a convex subset of $\mathcal{P}([0, 1]^2)$, the space of Borel probability measures on $[0, 1]^2$, and can be thought of as an infinite-dimensional analogue of the Birkhoff polytope. In contrast to the rich literature on the geometry of the Birkhoff polytope (e.g. \cite{BRUALDI1977194, Beck2003, Bouthat2025, Bouthat2025_2}), the metric geometry of $\Perm$ appears comparatively underexplored. In this paper we study the Chebyshev radius and centers of $\Perm$ with respect to the metric $\ds$. Informally, what is the smallest ball that covers $\Perm$ and which permutons might serve as centers for this ball of minimal size? 

\subsection{The main result}

The radius of the smallest ball containing a set of a metric space is called its Chebyshev radius, while any valid center of this ball is the Chebyshev center. We say that a permuton is \emph{1/2-periodic} if it is invariant under 1/2-translations along both coordinate axes (understood periodically).

\begin{theorem} \label{thm:main}
    The Chebyshev radius of $\Perm$ in $\ds$ is 1/4, and a permuton is a Chebyshev center if and only if it is 1/2-periodic.
\end{theorem}

\begin{remark} \label{remark:witnesses}
    Given Theorem \ref{thm:main}, one can easily find permutons witnessing this radius: a permuton $\nu$ is 1/4 apart from \emph{every} Chebyshev center if there exists a square $Q$ with sides $1/2$ and $\nu(Q)=1/2$. As we discuss through a chain of simple propositions in Section \ref{sec:examples}, these are the only universal witnesses, but some Chebyshev centers have further ones. We note that the existence of universal witnesses indicates a flat nature of the geometry.
\end{remark}

\subsection{Related literature}

While Chebyshev centers are a standard object in convex optimization and computational geometry, often under the name \emph{smallest enclosing ball} \cite{PAN200649, Mordukhovich2013}, they also appear in convex and metric geometry under related notions such as circumradius and Jung-type extremal problems on the ratio of the Chebyshev radius and the diameter \cite{Jung1901, Dekster1995}.

The Chebyshev center and radius of the Birkhoff polytope was studied in \cite{Bouthat2025, Bouthat2025_2} for the metric induced by $\ell^p$ norms and Schatten $p$-norms. Both these families consist of submultiplicative, permutation invariant matrix norms, and the unique Chebyshev center in both cases is the constant doubly stochastic matrix. While the finite dimensional analogue of $\ds$ is also induced by a norm defined by $\|A\| = \max_R \left|\sum_{(i, j)\in R}A_{i, j}\right|$, it is not permutation invariant nor submultiplicative. The lack of permutation invariance is key to the qualitatively different results.

We note that treating a permuton as a Markov kernel from $[0, 1]$ to $[0, 1]$, permutons are bounded linear operators from $L^p([0, 1])$ to itself. Thus in direct analogy to \cite{Bouthat2025} one could study the Chebyshev center and radius of $\Perm$ using the $L^p$ operator norm, however, this would not be natural from the permuton point of view as this norm does not induce the usual weak topology of permutons.

\subsection{The organization of the paper}
In Section \ref{sec:prelim}, we introduce notation and discuss preliminaries concerning 
basic properties of the rectangular metric and some measure theoretic tools. In Section \ref{sec:main} we prove a key lemma leading to the proof of the main theorem. In Section \ref{sec:examples} we discuss Remark \ref{remark:witnesses} in detail, providing a characterization of those $\nu\in\Perm$ for which $\ds(\nu, \mu)=1/4$ for a given Chebyshev center $\mu$.

\section{Preliminaries} \label{sec:prelim}

\subsection{Basic observations about the rectangular metric}

Put $\mathcal{R}$ for the closed axis-aligned rectangles. Denote by $\mathcal{M}([0, 1]^2)$ the set of signed measures on the unit square. This contains the space $\mathcal{P}([0, 1]^2)$; in particular, it contains $\Perm$. For any $\mu\in\mathcal{M}([0, 1]^2)$ on the unit square put 
$$\|\mu\|_{\square}=\max_{R\in\mathcal{R}}|\mu(R)|.$$
This is trivially a norm on the vector space $\mathcal{M}([0, 1]^2)$ and it induces $\ds$. This point of view simplifies certain ideas. For example
\begin{lemma}
    Let $(\mu_i)_{i=1}^{k}, (\nu_i)_{i=1}^{k}$ be elements of $\mathcal{M}([0, 1]^2)$, $\varepsilon>0$, and $\alpha_1, \dots, \alpha_k\geq 0$ with $\sum_{i=1}^{k}\alpha_i=1$.
    Then
    $$\ds(\mu_i, \nu_i) < \varepsilon, \quad i=1, 2, \dots k.$$
    implies
    $$\ds\left(\sum_{i=1}^k\alpha_i\mu_i, \sum_{i=1}^k\alpha_i\nu_i\right)<\varepsilon.$$
\end{lemma}
\begin{proof}
    $$\left|\sum_{i=1}^k\alpha_i\mu_i(R)-\sum_{i=1}^k\alpha_i\nu_i(R)\right|=\left|\sum_{i=1}^k\alpha_i(\mu_i-\nu_i)(R)\right|\leq \sum_{i=1}^{k}\alpha_i|\mu_i(R)-\nu_i(R)|<\varepsilon,$$
    where the last inequality follows from termwise bounds and $\sum_{i=1}^{k}\alpha_i=1$. Taking supremum concludes the proof.
\end{proof}

\subsection{A toric point of view}

Extending the sides of an axis-aligned rectangle $R\subseteq[0, 1]^2$ we obtain a partition of $[0, 1]^2$ to a geometry-dependent number of rectangles. However, after identifying the unit square $[0, 1]^2$ with the torus $\mathbb{T}^2=(\mathbb{R}/\mathbb{Z})^2$, for every non-degenerate $R$ (i.e., having both sidelengths $h, w$ strictly between 0 and 1) this partition consists of four rectangular pieces with sidelengths $(h, w)$, $(1-h, w)$, $(h, 1-w)$, $(1-h, 1-w)$. Thus every rectangle $R$ is naturally accompanied by three other axis-aligned \emph{toric rectangles}, whose set we denote by $\mathcal{R}_{\mathbb{T}^2}$. We will occasionally denote the rectangles in such a \emph{toric quartet} by $R_{0, 0}:=R, R_{0, 1}, R_{1, 0}, R_{1, 1}$, where $R_{i, j}$ and $R_{1-i, j}$ are sharing vertical sides and $R_{i, j}$ and $R_{i, 1-j}$ are sharing horizontal sides. Generically, three of these pieces are not standard rectangles included in the definition of $\ds$, however, including them does not change its value for permutons, as we establish it in the following proposition, making this point of view very convenient.

\begin{prop} \label{prop:torus_unit_square_identification}
    For any permutons $\mu, \nu$
    \begin{enumerate}[label=(\roman*)]
        \item{For any $k, l$ integers, with the addition understood mod 2
        $$\mu(R_{i, j})-\nu(R_{i, j})=(-1)^{k+l}\left(\mu(R_{i+k, j+l})-\nu(R_{i+k, j+l})\right).$$}\label{item1:prop:torus_unit_square_identification}
        \item The value of $|\mu(R_{i, j})-\nu(R_{i, j})|$ is independent of $i, j$.
        \item $$\ds(\mu, \nu)=\max_{R\in\mathcal{R}_{\mathbb{T}^2}}|\mu(R)-\nu(R)|.$$
    \end{enumerate}
\end{prop}

\begin{proof}
    By symmetry, for (i), it is sufficient to check what happens upon replacing $i$ by $i+1$. By the uniform marginal property,
    $$\mu(R_{i, j})+\mu(R_{i+1, j})=\text{height of $R_{i, j}$}=\nu(R_{i, j})+\nu(R_{i+1, j}),$$
    which can be rearranged in accordance with the statement.

    (ii) follows trivially from (i).

    For (iii), observe that for every $R\in \mathcal{R}_{\mathbb{T}^2}$, $R_{i, j}\in \mathcal{R}$ for some $i, j$, i.e., it is a standard rectangle. Using (ii), this implies that the supremum of differences over $\mathcal{R}_{\mathbb{T}^2}$ is the same as the supremum of differences over $\mathcal{R}$.
\end{proof}

On $\mathbb{T}^2$, between any two points there are infinitely many line segments. We would like to select one canonically. Parameterizing $\mathbb{T}^2$ by $[0, 1)^2$, we apply the convention that by the line segment $(x, y_1), (x, y_2)$ we mean the vertical line segments containing $(x, y_1+\varepsilon)$ for small enough $\varepsilon>0$. (That is, the role of the two endpoints is not symmetric.) We proceed analogously for horizontal line segments. Finally, a toric rectangle with opposite corners $(x_1, y_1), (x_2, y_2)$, is the unique rectangle with line segments  $[(x_1, y_1), (x_1, y_2)], [(x_1, y_1), (x_2, y_1)]$ as its sides. (Without such a specification, we could mean any rectangle in the same toric quartet.)

\subsection{Decomposition of measures}

To prove Theorem \ref{thm:main}, we will need some tools from measure theory describing the local magnitude of a measure.

We denote by $\lambda$ the $d$-dimensional Lebesgue measure. Given a $d$-dimensional Borel measure $\mu$, its Lebesgue decomposition with respect to $\lambda$ is
$\mu = \mu_{ac}+\mu^{\perp}$ with $\mu_{ac}\ll \lambda$, $\mu^{\perp} \perp \lambda$. Writing $f=\frac{d\mu_{ac}}{d\lambda}$ for the Radon--Nikodym derivative of $\mu_{ac}$ with respect to $\lambda$, the Lebesgue--Radon--Nikodym representation of $\mu$ with respect to $\lambda$ is understood as $d\mu = fd\lambda + d\mu^{\perp}$.

 Moreover, we say that a family $E_r$ of Borel subsets of $\mathbb{R}^d$ \emph{shrinks nicely} to $x\in\mathbb{R}^d$ if
\begin{itemize}
    \item $E_r\subseteq B(x, r)$ for each $r$, where $B(x, r)$ is the $r$-ball centered at $x$;
    \item there is a constant $\alpha>0$ such that for every $r$
    $$\lambda(E_r)\geq \alpha\lambda(B(x, r)).$$
\end{itemize}
Note that we do not require $x\in E_r$. For example, balls of any $p$-norm centered at the origin shrink nicely to origin, but so do their intersections with the positive orthant. (For $d=1$, it simply means half-intervals.) 

We use the following form of Lebesgue's differentiation theorem:

\begin{theorem}[{\cite[Theorem~3.22]{Folland}}]  \label{thm:lebesgue_differentiation}
    Let $\nu$ be a regular Borel measure, and let $d\nu=d\nu^{\perp}+fd\lambda$ be its Lebesgue--Radon--Nikodym representation. Then for $\lambda$-almost every $x\in \mathbb{R}^d$.
    \begin{equation}\label{eq:lebesgue_differentiation}
    \lim_{r\to 0}\frac{\nu(E_r)}{\lambda(E_r)}\to f(x)
    \end{equation}
    for every family $\{E_r\}_{r>0}$ which shrinks nicely to $x$.
\end{theorem}

We call such an $x$ a \emph{differentiability point} of $\mu$. We note that while the pointwise value $f(x)$ makes no sense a priori, in differentiability points it is natural and convenient to define it using \eqref{eq:lebesgue_differentiation}. Below we will use the differentiation theorem for permutons and certain one-dimensional measures associated with a given permuton, i.e., for $d=1, 2$.

\section{Proof of the main result} \label{sec:main}

The following lemma provides several equivalent geometric and metric characterizations of 1/2-periodicity, which will be crucial in identifying the Chebyshev centers.

\begin{lemma}\label{lemma:1/2_periodicity_equivalence}
    The following are equivalent for a permuton $\mu$ (below $w, h$ are always the sidelengths of the corresponding rectangle):
    \begin{enumerate}
        \item[(i)] $\mu$ is 1/2-periodic.
        \item[(ii)] for every toric rectangle $R$, if $h=1/2$ or $w=1/2$, then $\mu(R)=wh$.
        %\item for any standard rectangle $R$, if $h=1/2$ or $w=1/2$, then $\mu(R)=wh$.
        \item[(iii)] for every toric rectangle $R$ with $h+w\geq 1$ so that $w\leq 1/2$ or $h\leq 1/2$,
        $$\mu(R)\leq \frac{h+w}{2}-\frac{1}{4}.$$
        \item[(iv)] for every toric rectangle $R$ with $h+w=1$,
        $$\mu(R)\leq 1/4.$$
    \end{enumerate}
\end{lemma}

\begin{proof}
    \underline{\bf (i) $\Rightarrow$ (ii):} without loss of generality, assume $h=1/2$. For a rectangle $R$, put $R'=R+(0, 1/2).$ Then $R\cup R'$ equals a full vertical strip of width $w$, and $\mu(R)=\mu(R')$, thus $2\mu(R)=w$. (Due to $\mu$ having uniform marginals, we do not have to bother with the boundaries.)

    \underline{\bf (ii) $\Rightarrow$ (i):} due to space homogeneity and the symmetry of the coordinates, it suffices to prove that for every measurable set $A\subseteq[0, 1/2]^2$, we have
    $$\mu(A)=\mu(A+(0, 1/2))=:\mu'(A).$$
    As axis-aligned rectangles in $[0, 1/2]^2$ generate the Borel sigma-algebra of $[0, 1/2]^2$, it suffices to show that $\mu, \mu'$ coincides for axis-aligned rectangles, i.e., $\mu(R)=\mu(R')$ with the notation of the previous section.
    
    Consider a rectangle $R=[x_1, x_2]\times [y_1, y_2]\subseteq [0,1/2]^2$, define $R'$ as above and put $R^*=[x_1, x_2]\times [y_2, y_1+1/2]$. Then
    $$\mu(R\cup R^*)=\frac{x_2-x_1}{2}=\mu(R^*\cup R'),$$
    and thus $\mu(R)=\mu(R')$, as desired.

    \underline{\bf (ii) $\Rightarrow$ (iii):} without loss of generality, assume $h\leq 1/2$. Translating the measures and $R$ as well, we might assume that $R=[0, w]\times [0, h]$. Then we have
    $$R\subseteq \left([0, 1/2]\times [0, h]\right)\cup\left([1/2, w]\times [0, 1/2]\right),$$
    thus by (ii)
    \begin{equation}\label{eq:mu_of_R}
    \mu(R)\leq \frac{h}{2}+\frac{1}{2}\left(w-\frac{1}{2}\right)=\frac{h+w}{2}-\frac{1}{4}.
    \end{equation}

    \underline{\bf (iii) $\Rightarrow$ (ii):}:Without loss of generality, $R$ is such that $h=1/2$. Then by (iii), $\mu(R)\leq x/2$, and taking its toric quartet, also $\mu(R_{1, 0})\leq \frac{1-x}{2}$. However, by uniform marginals, the sum of these inequalities must be saturated, thus actually both of them holds with equality.

    \underline{\bf (iii) $\Rightarrow$ (iv):} Immediate, (iv) is a special case of (iii).

    \underline{\bf (iv) $\Rightarrow$ (i):}
    For any $x, y\in [0, 1]$, consider the one-dimensional measures $\mu_{x+}, \mu_{x-}, \mu^{y+}, \mu^{y-}$ defined by
    $$\mu_{x-}(A)=\mu([x-1/2, x]\times A), \qquad \mu_{x+}(A)=\mu([x, x+1/2]\times A).$$
    $$\mu^{y-}(A)=\mu(A\times [y-1/2, y]), \qquad \mu^{y+}(A)=\mu(A\times [y, y+1/2]).$$
    Note that by applying \cref{thm:lebesgue_differentiation} to $\mu_{x\pm}, \mu^{y\pm}, \mu$ respectively, we have
    \begin{itemize}
        \item for every $x$ Lebesgue almost every $y$ is a differentiability point of $\mu_{x\pm}$;% i.e., $y\in\mathcal{D}(\mu_x)$;
        \item for every $y$ Lebesgue almost every $x$ is a differentiability point of $\mu^{y\pm}$;%i.e., $x\in\mathcal{D}(\mu^y)$;
        \item Lebesgue almost every $(x, y)$ is a differentiability point of $\mu$. % i.e., $(x,y)\in\mathcal{D}(\mu)$.
    \end{itemize}
    Applying Fubini's theorem for the first two pairs of these observations, we get that Lebesgue almost every $(x, y)$ satisfies all these conclusions. Denote the full measure set of such points by $D$.
    
    As $\mu$ has uniform marginals, $\mu_{x\pm}$ and $\mu^{y\pm}$ are absolutely continuous with respect to the 1-dimensional Lebesgue measure. Denote their densities by $f_{x\pm}, f^{y\pm}$ respectively. Again due to the uniform marginal property,
    $$f_{x-}+f_{x+}=1, f^{y-}+f^{y+}=1 $$
    for almost every $x, y$, with the equalities holding pointwise for points of $D$.
    
    If both $f_{x+}\equiv1/2, f^{y+}\equiv 1/2$ almost everywhere for almost every choice of $x, y$, then we have $\mu_x(A)=\mu^{y}(A)=\frac{\lambda(A)}{2}$ for every $x, y$, and hence $\mu$ is 1/2-periodic, we are done. Assume that it is not the case. Then without loss of generality, for some $X$ with $\lambda(X)>0$, for every $x\in X$ we have that $f^{y+}(x)\neq 1/2$ for $y\in Y_x$ with $\lambda(Y_x)>0$. 
    %As for every $x$ almost every $y$ is a differentiability point of $\mu_x$, we can even assume that each $Y_x$ consists of differentiability points of $\mu_x$ exclusively. 
    Then for $$B=\bigcup_{x\in[0,1]} \{x\}\times Y_{x}\subseteq [0, 1]^2,$$
    $\lambda(B)>0$. Even $\lambda(D\cap B)>0$.

    Pick a point $(x, y)\in D\cap B$, without losing generality assume that $f^{y+}(x)>1/2$. Either $f_{x-}(y)$ or $f_{x+}(y)$ is at most 1/2, without loss of generality assume $f_{x+}(y)\leq 1/2$. Consider the line segments
    $$I=[(x, y), (x+1/2, y)], \quad J=[(x, y), (x, y+1/2)].$$
    
    Let $R$ be the toric rectangle with both sidelengths equal to 1/2 and having sides $I, J$. For some $\varepsilon>0$ to be fixed later, denote moreover by $R_{+}$ the $\varepsilon$-width vertical strip adjacent to $J$ outside of $R$ and by $R_{-}$ the $\varepsilon$-height horizontal strip adjacent to $I$ inside of $R$. (See Figure \ref{fig:chebyshev_center}.)
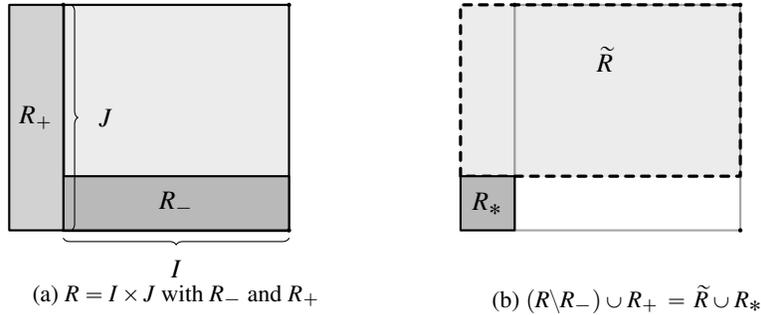
\begin{figure}[h]
\centering
\begin{tikzpicture}[x=6cm,y=6cm,>=Latex, line cap=round, line join=round]

% Parameters
\def\eps{0.12} % epsilon as fraction of the unit (relative to the 1x1 canvas)
\def\half{0.5}

% ---------- LEFT PANEL ----------
\begin{scope}[shift={(0,0)}]

  \fill[gray!15] (0,0) rectangle (\half,\half);
  \draw[thick] (0,0) rectangle (\half,\half);

  \draw[very thick] (0,0) -- (\half,0);%node[midway,below=2pt] {$I$};
  \draw[very thick] (0,0) -- (0,\half);%node[midway,left=2pt] {$J$};

  \fill[gray!55] (0,0) rectangle (\half,\eps);
  \draw[thick] (0,0) rectangle (\half,\eps);
  \node at ({0.25}, {0.5*\eps}) {$R_{-}$};

  \fill[gray!35] (-\eps,0) rectangle (0,\half);
  \draw[thick] (-\eps,0) rectangle (0,\half);
  \node at ({-0.5*\eps}, {0.25}) {$R_{+}$};
\draw[decorate,decoration={brace,mirror,raise=3pt}]
  (0,0) -- (\half,0)
  node[midway,below=8pt] {$I$};

\draw[decorate,decoration={brace,mirror, raise=3pt}]
  (0,0) -- (0,\half)
  node[midway,right=10pt] {$J$};

  \fill (\half,0) circle (0.8pt);
  \fill (0,\half) circle (0.8pt);
  \fill (\half,\half) circle (0.8pt);

  \node[below] at ({0.25}, {-0.10}) {\small (a) $R=I\times J$ with $R_-$ and $R_+$};
\end{scope}

% ---------- RIGHT PANEL ----------
\begin{scope}[shift={(1,0)}]

  \fill[gray!15] (-\eps,\eps) rectangle (\half,\half);

  \draw[very thick, dashed] (-\eps,\eps) rectangle (\half,\half);
  \node at ({0.20}, {0.38}) {$\widetilde{R}$};

  \fill[gray!55] (-\eps,0) rectangle (0,\eps);
  \draw[thick] (-\eps,0) rectangle (0,\eps);
  \node at ({-0.5*\eps}, {0.5*\eps}) {$R_{\ast}$};

  \draw[thick, opacity=0.35] (0,0) rectangle (\half,\half);

  \fill (\half,0) circle (0.8pt);
  \fill (0,\half) circle (0.8pt);
  \fill (\half,\half) circle (0.8pt);

  \node[below] at ({0.25}, {-0.10}) {\small (b) $(R\setminus R_-)\cup R_+ \,=\, \widetilde{R}\cup R_\ast$};
\end{scope}

\end{tikzpicture}
\caption{Rectangles in the Chebyshev-center argument.}
\label{fig:chebyshev_center}
\end{figure}
    
    Then modulo the null-measure boundaries,
    $$(R\setminus R_{-})\cup R_{+}=\widetilde{R}\cup R_\ast,$$
    where $\widetilde{R}$ is the rectangle with opposite corners $(x-\varepsilon, y+\varepsilon), (x+1/2, y+1/2)$ and $R_\ast$ is the rectangle with opposite corners $(x-\varepsilon, y), (x, y+\varepsilon)$. Thus as $R_{-}\subseteq R$ and all these unions are non-overlapping, we have
    \begin{equation} \label{eq:mu_of_Rtilde}
        \mu(\widetilde{R})=\mu(R)-\mu(R_{-})+\mu(R_{+})-\mu(R_\ast). 
    \end{equation}

    We know $\mu(R)=1/4$, otherwise some rectangles in the toric quartet of $R$ would have measure exceeding 1/4, contradicting (iv).
    
    As $y$ a differentiability point of $\mu_{x+}$ and $f_{x+}(y)\leq 1/2$, by definition we know that 
    $$\mu(R_{-})\leq \frac{\varepsilon}{2}+o(\varepsilon).$$
    Similarly, for some $c>0$ with all small enough $\varepsilon>0$
    $$\mu(R_{+})\geq \left(\frac{1}{2}+c\right)\varepsilon.$$
    
    Finally, as $(x, y)$ is a differentiability point of $\mu$,
    $$\mu(R_\ast)=O(\varepsilon^2).$$
    Plugging in all these observations into \eqref{eq:mu_of_Rtilde}, we find that for small enough $\varepsilon$
    $$\mu(\widetilde{R})>1/4,$$
    contradicting (iv). Thus $f_{x-}, f_{x+}, f^{y-}, f^{y+}$ indeed equal 1/2 almost everywhere, concluding the proof.
    
\end{proof}

\begin{proof}[Proof of Theorem \ref{thm:main}]
    Take any permutons $\mu, \nu$ and any toric rectangle $R\in\mathcal{R}_{\mathbb{T}^2}$. Consider the corresponding toric quartet. If $\mu(R)-\nu(R)>1/2$, then by \cref{prop:torus_unit_square_identification} (i), we have $\mu(R_{1, 1})-\nu(R_{1, 1})>1/2$, which implies $\mu([0, 1]^2)>1$, an obvious contradiction. Thus the diameter is at most 1/2. Moreover, we know that $\ds(Id, 1-Id)=1/2$, as demonstrated by $R=[0, 1/2]^2$, implying that the diameter is actually 1/2.

    Concerning the Chebyshev center, as the diameter is 1/2, we know that the Chebyshev radius cannot be smaller than 1/4. 

    Assume now that $\mu$ is 1/2-periodic, and hence by Lemma \ref{lemma:1/2_periodicity_equivalence} $\mu(R)=hw$ for every rectangle with sidelengths $h, w$ if $h=1/2$ or $w=1/2$. Now it is sufficient to show that it is a Chebyshev center. Proceeding towards a contradiction, assume the existence of a permuton $\nu$ and a toric rectangle $R$ for which 
    \begin{equation}\label{eq:towards_contradiction}
    \mu(R)-\nu(R)>1/4.
    \end{equation}
    Switching to $R_{1, 1}$ if necessary, by Proposition \ref{prop:torus_unit_square_identification} (i), we can assume $h\leq 1/2$. Observe furthermore that if $w\leq 1/2$ as well, then by 1/2-periodicity $\mu(R)\leq 1/4$, ruling out this possibility. Hence $w>1/2$ can be assumed.

    Translating the measures and $R$ as well, we might assume that $R=[0, w]\times [0, h]$. Then we have
    $$R\subseteq \left([0, 1/2]\times [0, h]\right)\cup\left([1/2, w]\times [0, 1/2]\right),$$
    thus by 1/2-periodicity
    \begin{equation}\label{eq:mu_of_R}
    \mu(R)\leq \frac{h}{2}+\frac{1}{2}\left(w-\frac{1}{2}\right)=\frac{h+w}{2}-\frac{1}{4}.
    \end{equation}
    Considering \eqref{eq:towards_contradiction}, this implies $h+w> 1$.
    
    Consider now $R_{1, 0}=[w, 1]\times [0, h]$. By uniform marginals, we have
    $$\nu(R_{1, 0})\leq 1-w, \qquad \nu(R_{1, 0}\cup R)=h,$$
    thus
    $$\nu(R)\geq h-(1-w)=h+w-1.$$
    Pairing this with \eqref{eq:mu_of_R}, we obtain
    $$\mu(R)-\nu(R)\leq \frac{3}{4}-\frac{h+w}{2}<\frac{1}{4},$$
    as we already noted $h+w>1$. This ultimately contradicts \eqref{eq:towards_contradiction}. Thus indeed every $1/2$-periodic permuton $\mu$ is a Chebyshev center of $\Perm$ with respect to $\ds$.

    Finally, we prove that if $\mu$ is not 1/2-periodic then it cannot be a Chebyshev center. This follows immediately from (iv) of \cref{lemma:1/2_periodicity_equivalence}: the lack of 1/2-periodicity implies the existence of $R$ with sidelengths $w, 1-w$ and $\mu(R)>1/4$. However, for such $R$, the uniform measure $\nu$ on $R_{0, 1}\cup R_{1, 0}$ is a permuton for which $\nu(R)=0$. Thus $\ds(\mu, \nu)>1/4$.

\end{proof}

\section{Witnesses for the Chebyshev radius} \label{sec:examples}

In this section, we briefly discuss the content of Remark \ref{remark:witnesses}.
We are interested in which permutons $\nu$ satisfy $\ds(\nu, \mu)=1/4$ for a given Chebyshev center $\mu$.

\begin{prop} \label{prop:universal_witnesses}
Let $\nu\in \Perm$. There exists $\nu'$ with $\ds(\nu, \nu')=1/2$ if and only if there exists a toric square $Q\subseteq [0, 1]^2$ with side length 1/2 and $\nu(Q)=1/2$.
\end{prop}

\begin{proof}
    Assume first that such a $Q$ exists, consider its toric quartet $Q_{0,0}, Q_{0, 1}, Q_{1, 0}, Q_{1,1}$. Then by uniform marginals, $\supp \nu \subseteq Q_{0, 0}\cup Q_{1, 1}$, and we can choose $\nu'$ to be the uniform measure on $Q_{0, 1}\cup Q_{1, 0}$.

    For the other direction choose $R$ to be the witness of $\ds(\nu, \nu')=1/2$. Passing to another member of the toric quartet, we can assume that both side lengths of $R$ are at most 1/2, and if either of them is smaller than 1/2, $\nu(R),\nu'(R)<1/2$, ruling out $|\nu(R)-\nu'(R)|=1/2$. Thus $R$ is a square with side lengths 1/2, and one of the squares in its toric quartet is a proper choice for $Q$.
\end{proof}

Any permuton $\nu$ satisfying the assumption of Proposition \ref{prop:universal_witnesses} obviously has the property that $\ds(\nu, \mu)=1/4$ for every Chebyshev center $\mu$ by a triangle inequality. We call such a $\nu$ a \emph{trivial witness}. The following theorem quickly implies that only trivial witnesses are universal witnesses, i.e., only these witness the Chebyshev radius for every Chebyshev center.

\begin{theorem}\label{thm:only_trivial_witnesses_equivalence}
    The following are equivalent for a Chebyshev center $\mu\in \Perm$:
    \begin{enumerate}
        \item[(i)] $\mu$ has only trivial witnesses.
        \item[(ii)] For a toric rectangle $R$ with side lengths $h+w=1$, we have $\mu(R)=1/4$ only when $w=1/2$.
        \item[(iii)] For every toric square $Q$, $\mu(Q)>0$.
    \end{enumerate}
\end{theorem}

Note that (ii) is equivalent to saying that the description of Chebyshev centers given by (iv) of Lemma \ref{lemma:1/2_periodicity_equivalence} holds with equality only for $h=w=1/2$.

\begin{remark}
    As (iii) of Theorem  \ref{thm:only_trivial_witnesses_equivalence} obviously holds for $\mu=\lambda$, only trivial witnesses are universal indeed.
\end{remark}

\begin{proof}[Proof of Theorem \ref{thm:only_trivial_witnesses_equivalence}]
\underline{\bf (ii) $\Rightarrow$ (i):} in the proof of Theorem \ref{thm:main} consider \eqref{eq:towards_contradiction} and the subsequent discussion, this time starting with the alternative assumption that $\mu(R)-\nu(R)=1/4$ for a rectangle $R$ with side lengths $h,w$. In that proof, we see that $h\leq 1/2$ can be assumed and then one 
    obtains 
    \begin{equation} \label{eq:muR_upper_bound}
    \mu(R)\leq \frac{h+w}{2}-\frac{1}{4},
    \end{equation}
    from which $h+w\geq 1$ is deduced. Moreover,
    $$\mu(R)-\nu(R)\leq \frac{3}{4}-\frac{h+w}{2}.$$
    This is compatible with $\mu(R)-\nu(R)=1/4$ only if $h+w=1$, but then \eqref{eq:muR_upper_bound} implies $\mu(R)\leq \frac{1}{4}$. Thus actually $\mu(R)=\frac{1}{4}, \nu(R)=0$. By assumption, this implies $h=w=1/2$, $\nu$ is a trivial witness. 
    
\underline{\bf (i) $\Rightarrow$ (ii):} if $\mu(R)=1/4$ for $R$ with side lengths $h+w=1$, then the uniform measure on $R_{0, 1}\cup R_{1, 0}$ is a permuton 1/4 apart from $\mu$. However, it is not a trivial witness once $w\neq 1/2$.

\underline{\bf (ii) $\Rightarrow$ (iii):} Recall the configuration of Figure \ref{fig:chebyshev_center} with $R_{*}=Q$. Consider the computation \eqref{eq:mu_of_Rtilde}. Once $\mu$ is a Chebyshev center, we have $\mu(R_{-})=\mu(R_{+})$, and hence $\mu(\widetilde{R})=\mu(R)-\mu(Q)=1/4-\mu(Q)$. By (ii), $\mu(\widetilde{R})<1/4$ and hence $\mu(Q)>0$.

\underline{\bf (iii) $\Rightarrow$ (ii):} We can repeat the previous argument focusing on Figure \ref{fig:chebyshev_center}, this time choosing $R=\widetilde{R}$.
\end{proof}

From the previous proof, we also obtain the following characterization:

\begin{theorem}
    For a Chebyshev center $\mu$, the permuton $\nu$ is a witness of the extremal distance $\ds(\mu, \nu)=1/4$ if and only if there exists a toric rectangle $R$ with side lengths $h+w=1$ such that $\mu(R)=1/4, \nu(R)=0$.
\end{theorem}

Although this characterization has a similar nature to the definition itself, it is useful as it provides a simpler point of view by significantly restricting which rectangles should be considered as maximizers in the definition of $\ds$.

\begin{example}[Singular Chebyshev centers with and without nontrivial witnesses]
    Consider the functions $f_1(x)=x, f_2(x)=x+1/2 \mod 1$ and let $\mu$ be the uniform measure on the union of their graphs. Then $\mu$ is 1/2-periodic, hence it is a Chebyshev center. However, (iii) of Theorem \ref{thm:only_trivial_witnesses_equivalence} is trivially violated, implying the existence of nontrivial witnesses. More explicitly, observe that for any $\frac{1}{2}\leq w\leq \frac{3}{4}$, we have $\mu(R)=1/4$ for $R=[0, w]\times [0, 1-w]$. Now if $\nu(R)=0$ then $\ds(\nu, \mu)=1/4$, and typically $\nu$ is not a trivial witness. For example, $\nu$ being the uniform measure on $R_{0, 1}\cup R_{1, 0}$ is a suitable choice.

    In contrast, if we consider $f_n(x)=x+\alpha_n \mod 1$ for a dense subset $\{\alpha_n\}_{n=1}^{\infty}\subseteq [0, 1]$ and consider the permuton which puts mass $2^{-n}$ uniformly on the graph of $f_n$, we obtain a singular Chebyshev center satisfying (iii) of Theorem \ref{thm:only_trivial_witnesses_equivalence}, thus having only trivial witnesses.
\end{example}

\subsection*{Acknowledgements}

The author is thankful to Boglárka Gehér and Ágnes Cs. Kúsz for the helpful discussion.

\printbibliography

\end{document}